\documentclass[10pt,british]{article}
\usepackage{amsfonts}
\usepackage{latexsym}
\usepackage{amsmath}
\usepackage{amssymb}
\usepackage{amssymb}
\usepackage{amsthm}
\usepackage{mathrsfs}
\usepackage[pagebackref,colorlinks=true]{hyperref}

\newtheorem{thrm}{Theorem}[section]

\newtheorem{lemma}[thrm]{Lemma}

\newtheorem{prop}[thrm]{Proposition}

\newtheorem{cor}[thrm]{Corollary}

\newtheorem{dfn}[thrm]{Definition}

\newtheorem{rmrk}[thrm]{Remark}

\newtheorem{conv}[thrm]{Convention}

\newtheorem{exam}[thrm]{Example}

\newcommand{\newsection}{    
\setcounter{equation}{0}\section}
\def\appendix#1{\addtocounter{section}{1}\setcounter{equation}{0}
\renewcommand{\thesection}{\Alph{section}}
\section*{Appendix \thesection\protect\indent \parbox[t]{11.15cm}{#1}}
\addcontentsline{toc}{section}{Appendix \thesection\ \ \ #1}}

\newcommand{\be}{\begin{eqnarray}}
\newcommand{\ee}{\end{eqnarray}}
\newcommand{\bea}{\begin{eqnarray}}
\newcommand{\eea}{\end{eqnarray}}
\newcommand{\ba}{\begin{array}}
\newcommand{\ea}{\end{array}}

\newenvironment{frcseries}{\fontfamily{frc}\selectfont}{}

\def\d{\delta}

\def\sb {{\nabla}}
\def\LC{{\nabla^g}}
\def\p{{\varphi}}
\def\ph{{\Phi}}
\def\ps{{\Psi^+}}
\def\sp{{\Psi^-}}

\begin{document}
\begin{center}
\vspace*{-1.0cm}


\vspace{1.2 cm} {\Large \bf The Riemannian curvature identities of a $G_2$ connection \\[3mm]  with  skew-symmetric torsion and generalized Ricci solitons
} 

\vspace{1cm}
 {\large S.~ Ivanov${}^1$ and  N. Stanchev$^2$}

\vspace{0.5cm}

${}^1$ University of Sofia, Faculty of Mathematics and
Informatics,\\ blvd. James Bourchier 5, 1164, Sofia, Bulgaria
\\and  Institute of Mathematics and Informatics,
Bulgarian Academy of Sciences

\vspace{0.5cm}
${}^2$ University of Sofia, Faculty of Mathematics and
Informatics,\\ blvd. James Bourchier 5, 1164, Sofia, Bulgaria\\

\vspace{0.5cm}

\end{center}

\begin{abstract}
Curvature properties of the characteristic connection on an integrable $G_2$ manifold  are investigated.   We consider integrable $G_2$ manifold of constant type, i.e. the scalar product of the exterior derivative of the $G_2$ form with its Hodge dual is a constant. We show that on an integrable $G_2$ manifold of constant type with $G_2$-instanton characteristic curvature  and vanishing Ricci tensor the torsion 3-form is harmonic. Consequently, we prove that the characteristic curvature is symmetric in exchange the first and the second pair and Ricci flat  if and only if the three-form torsion is  parallel with respect to the Levi-Civita and to the characteristic connection simultaneously and this is equivalent to the condition that the characteristic curvature   satisfies the Riemannian first Bianchi identity. We find that the Hull connection is a $G_2$-instanton exactly when the torsion is closed. We observe  that any compact integrable $G_2$ manifold with closed torsion is a generalized gradient Ricci soliton and this is equivalent to a certain vector field to be parallel with respect to the characteristic connection. In particular, this vector field is  an infinitesimal automorphism of the $G_2$ structure and preserves the torsion three form.

\medskip

AMS MSC2010: 53C55, 53C21, 53D18, 53Z05
\end{abstract}


\tableofcontents

\setcounter{section}{0}
\setcounter{subsection}{0}



\newsection{Introduction}
Riemannian manifolds with metric connections having totally skew-symmetric torsion and special holonomy received a lot of interest in mathematics and theoretical physics mainly from supersymmetric string theories and supergravity.  The main reason comes from the Hull-Strominger system which describes the supersymmetric background in heterotic string theories \cite{Str,Hull}. The number of preserved supersymmetries depends on the number  of  parallel spinors with respect to a metric connection $\sb$ with totally skew-symmetric torsion $T$. 
The existence  of a $\nabla$-parallel spinor leads to a restriction of the
holonomy group $Hol(\nabla)$ of the torsion connection $\nabla$.
Namely, $Hol(\nabla)$ has to be contained in $SU(n), dim=2n$, $Sp(n), dim=4n$ 
\cite{Str,GMW,sethi,IP1,IP2,Car,BB,BBE,GIP}, the exceptional group $G_2,
dim=7$ \cite{FI,GKMW,FI1}, the Lie group $Spin(7), dim=8$
\cite{GKMW,I1}. A detailed analysis of the possible geometries is
carried out in \cite{GMW}. 

The Hull-Strominger system for $SU(n)$ or $Sp(n)$ holonomy has been investigated intensively , see e.g.  \cite{LY,yau,yau1,FIUV,XS,OLS,PPZ,PPZ1,PPZ2,PPZ3,PPZ4,FHP,FHP1,CPYau,CPY1,Ph} and references therein.

The Hull-Strominger system in dimension seven considered in \cite{CGFT,Oss2} is known as  the $G_2$-Strominger system or heterotic $G_2$ system \cite{Oss2}. It  consists of the supersymmetry  equations and the anomaly cancellation condition. The latter expresses the exterior derivative of the 3-form torsion in terms of a difference of the first Pontrjagin forms of an $G_2$ instanton connection on an auxiliary vector bundle and a connection on the tangent bundle. The extra requirements for a solution of the supersymmetry equations  and the anomaly cancellation condition to provide  a supersymmetric vacuum of the theory is given by the $G_2$ instanton condition on the connection on the tangent bundle \cite{Iv} (see also \cite{MS,XS}). The $G_2$ instanton condition means that the curvature 2-form belongs to the Lie algebra $\frak{g}_2$ of the Lie group $G_2$. In general, 
 Hull \cite{Hull}  used the more physically accurate Hull connection   to define
the first Pontrjagin form on the tangent bundle. However, this choice leads to a system of equations, which is not
mathematically closed: e.g. the curvature of the Hull connection is only an  instanton modulo higher order  corrections, see \cite{MS}.  In this spirit, it  seems interesting to investigate when the Hull connection is a $G_2$ instanton. 


Necessary and sufficient  conditions for a $G_2$ structure  $\p$ to admit a metric connection with torsion 3-form preserving the $G_2$ structure are found in  \cite{FI}, namely the $G_2$ structure has to be integrable,  $d*\p=\theta\wedge*\p$, where $\theta$ is the Lee form defined below in \eqref{g2li} (see also \cite{GKMW,FI1,GMW,GMPW,II}). The $G_2$ connection constructed in \cite[Theorem~4.8]{FI} is unique and it is called  \emph{the characteristic connection}. 

From the point of view of physics, the $G_2$-Strominger system is a particular instance of a more
general system of equations, known as the Killing spinor equations in (heterotic) supergravity.
The compactification of the physical theory leads to the study of models of the form $N^k\times M^{10-k}$, where $N^k$ is a $k$-dimensional Lorentzian manifold and $M^{10-k}$ is a Riemannian spin
manifold which encodes the extra dimensions of a supersymmetric vacuum. 
 For application to the $G_2$-Strominger system, the integrable $G_2$ structure should be \emph{strictly integrable}, i.e. the scalar product $(d\p,*\p)=0$,  and the Lee form $\theta$ has to be an exact form, \cite{GKMW}.   It should be mentioned that strictly integrable  $G_2$ structure with an exact Lee form enforce $N=R^{1,2}$ in the compactification. A different compactification ansatz, with $N$ anti-de Sitter space-time, leads to a more general class of solutions with $(d\p,*\p)=\lambda=const.$ \cite{OLS} and the constant  $(d\p,*\p)$ is interpreted as the AdS radius \cite{Oss1,Oss2} see also  \cite[Section~5.2.1]{AMP}. We call this class \emph{integrable $G_2$ structure of constant type}.  Compact solutions to the heterotic $G_2$ system are constructed in \cite{FIUV1,OLM,LE}.  A geometric flow point of view on the  $G_2$-Strominger system in dimension seven has been developed recently in \cite{AMP}.
 
Special attention is also paid when the torsion 3-form $T$  is closed. For example, in type II string theory, $T$ is identified with the 3-form field strength. This is required by construction to satisfy $dT=0$  (see e.g. \cite{GKMW,GMW}).  
More generally, the geometry of a torsion connection with closed torsion form appears  in the framework of the generalized Ricci flow and the  generalized (gradient) Ricci solitons developed by Garcia-Fernandez and Streets \cite{GFS} (see  the references therein).

 The main purpose of the paper is to develop  curvature properties of the characteristic connection on 7-dimensional integrable $G_2$ manifold and to find necessary conditions for the integrable $G_2$ structure to be of constant type. We investigate  when the   characteristic and  the Hull connections have curvature which is a $G_2$ instanton. We consider the problem when the condition  $R\in S^2\Lambda^2$ of the curvature $R$ of the characteristic connection  implies the validity of the Riemannian first Bianchi identity \eqref{RB}.  
 
In what follows, we call  the curvature $R$ of the characteristic connection  \emph{the characteristic curvature } and the Ricci tensor $Ric$ of the characteristic connection,\emph{the characteristic Ricci tensor}.
 
Our first aim  is to show the following
\begin{thrm}\label{inst}
Let $(M,\p)$ be  an integrable $G_2$ manifold of constant type and the  curvature of the characteristic  connection $\sb$   is a Ricci flat  $G_2$-instanton, i.e.
\[d*\p=\theta\wedge*\p \qquad (d\p,*\p)=const.,\quad R\in \frak{g}_2\otimes\frak{g}_2, \quad Ric=0.\]
Then the torsion 3-form is harmonic, $\delta T=dT=0$, and the covariant derivatives of the 3-form $T$ with respect to the Levi-Civita connection and the characteristic connection coincide, $\LC T=\sb T$.
\end{thrm}
As a consequence of Theorem~\ref{inst}, we obtain
\begin{thrm}\label{mainsu3}
On an integrable $G_2$ manifold of constant type, the following conditions are equivalent:
\begin{itemize}
\item[a)] The characteristic connection  has curvature $R \in S^2\Lambda^2$ with vanishing characteristic Ricci tensor;
\item[b)] The curvature of the characteristic connection satisfies the Riemannian first Bianchi identity \eqref{RB};
\item[c)] The torsion 3 form  is  parallel  with respect to the Levi-Civita and to the characteristic connections simultaneously, $\LC T=\sb T=0$.
\end{itemize}
In these cases 
the exterior derivative $d\p$ of the $G_2$-form $\p$ is  $\sb$-parallel, $\sb(d\p)=0$.
\end{thrm} 
Concerning the Hull connection, we show in Theorem~\ref{hul} that it is a $G_2$ instanton  exactly when the  torsion 3-form is closed. This may have  applications  also   in type II string theories.

Integrable $G_2$ structures  with parallel torsion 3-form with respect to the characteristic connection are investigated in  \cite{F,AFer,CMS}   and a large number of examples are given there. 
Note that integrable $G_2$ structures  with $\sb$-parallel torsion 3-form  have co-closed Lee form. More generally,   due to  \cite[Theorem~3.1]{FI1},  for any  integrable $G_2$ structure  on a compact manifold there exists a unique integrable $G_2$ structure conformal to the original one with co-closed Lee form,  called \emph{the Gauduchon $G_2$ structure}.

It is known that if the $G_2$ structure $\p$ is co-calibrated, $d*\p=0$ then the characteristic Ricci tensor 
vanishes if and only if the torsion 3-form is harmonic  \cite[Theorem~5.4]{FI}. In this case, the scalar product $(d\p,*\p)$ is constant, i.e. the  integrable $G_2$ manifold is  of constant type.  Co-calibrated  $G_2$ structures  with vanishing characteristic Ricci tensor are investigated in detail in \cite{Fr}.

We extend the above result as follows
 \begin{thrm}\label{cooo}
Let $(M,\p)$ be a compact integrable $G_2$ manifold with a Gauduchon $G_2$ structure.  

If the characteristic Ricci tensor $Ric$ is symmetric and nonnegative then  $Ric=0$, the torsion 3-form is closed, co-closed with constant norm, the $G_2$ structure is of constant type with $\sb$-parallel Lee form.
\end{thrm}

 The characteristic Ricci tensor was computed in terms of the exterior derivative $dT$ and the covariant derivative $\sb T$ of the torsion 3-form $T$ in \cite[Theorem~5.1]{FI}. An immediate consequence is that if the torsion is closed then the characteristic Ricci tensor $Ric$ is equal to minus the covariant derivative of the Lee form. 
We observe   that the converse is also true and in this case, the integrable $G_2$ structure is of constant type.  This helps to show the validity of the next result.
\begin{thrm}\label{closTt}
Let $(M,\p)$ be a  compact  integrable $G_2$ manifold with closed torsion, $dT=0$. 

The following conditions are equivalent:
\begin{itemize}
\item[a).] The characteristic  connection  is Ricci flat, $Ric=0$;
\item[b).] The $G_2$ structure is a Gauduchon $G_2$ structure, $\delta\theta=0$;
\item[c).] The norm of the torsion is constant, $d||T||^2=0$;
\item[d).] The Riemannian scalar curvature is a non-negative constant, $Scal^g=const.\ge 0$.
\end{itemize}
In each of these cases,  the torsion is a harmonic 3-form.
\end{thrm}
As a consequence of Theorem~\ref{closTt}, we get
\begin{cor}\label{parallel}
A compact integrable $G_2$ manifold $(M,\p)$ with closed torsion and zero Riemannian scalar curvature is parallel, $\LC\p=0$.
\end{cor}

We show in Theorem~\ref{inf} that any compact integrable $G_2$ manifold with closed torsion 3-form is a steady generalized gradient Ricci soliton  and this condition is equivalent to a certain vector field to be parallel with respect to the characteristic connection. We also find out that this  vector field  is an infinitesimal automorphism of the $G_2$  structure  and preserves the torsion three form. 

Viewing integrable $G_2$ structures with closed torsion (equivalently, with an $G_2$ instanton Hull connection) as steady generalized Ricci solitons supplies, in general, a parallel vector field V which preserves the $G_2$ structure and the torsion 3-form T. These facts could help to understand more precisely the transversal SU(3) geometry (c.f. \cite{OLM})  
and  may reflect to a possible more precise description of the structure of integrable $G_2$ manifolds with closed  torsion, and could be a subject of a subsequent paper. In particular, the non-characteristic-flat compact example $SU(2)\times K3$ described in \cite{PP,FF} (see Example~\ref{noflate}) is an example of a steady generalized Ricci soliton.

\begin{rmrk} We recall \cite[Theorem~4.1]{AFF} which states that an irreducible complete and simply connected Riemannian manifold of dimension bigger or equal to $5$ with $\sb$-parallel and closed torsion 3-form, $\sb T=dT=0$ (which is equivalent to $\sb T=\sigma^T=0$ due to \eqref{sigma} and \eqref{dh}  below), is a simple compact Lie group or its dual non-compact symmetric space with biinvariant metric, and, in particular the torsion connection is the flat Cartan connection. In this spirit, our results above imply that  the irreducible complete and simply connected case in Theorem~\ref{mainsu3} cannot occur since the $G_2$ manifold should be a simple  compact Lie group of dimension seven but it is well known that there are no such groups. 
\end{rmrk}
We remark that similar investigations were done also for Spin(7) manifold in \cite{IP}.
\begin{conv}
Everywhere in the paper we will make no difference between tensors and the corresponding forms via the metric as well as we will  use Einstein summation conventions, i.e. repeated Latin  indices are summed over.

\end{conv}

\section{Preliminaries}
In this section, we recall some known curvature properties of a metric connection with totally skew-symmetric torsion on a Riemannian manifold as well as 
the notions and existence of a metric  connection preserving a given $G_2$ structure and having totally skew-symmetric torsion from \cite{I,FI,IS}. 
\subsection{Metric connection with skew-symmetric  torsion and its curvature}
On a Riemannian manifold $(M,g)$ of dimension $n$ any metric connection $\sb$ with totally skew-symmetric torsion $T$ is connected with the Levi-Civita connection $\sb^g$ of the metric $g$ by
\begin{equation}\label{tsym}
\sb^g=\sb- \frac12T \quad leading\quad to \quad \LC T=\sb T+\frac12\sigma^T,
\end{equation}
where the 4-form $\sigma^T$, introduced in \cite{FI}, is defined by
 \begin{equation}\label{sigma}
 \sigma ^T(X,Y,Z,V)=\frac12\sum_{j=1}^n(e_j\lrcorner T)\wedge(e_j\lrcorner T)(X,Y,Z,V), 
\end{equation} 
   $(e_a\lrcorner T)(X,Y)=T(e_a,X,Y)$ is the interior multiplication and $\{e_1,\dots,e_n\}$ is an orthonormal  basis.

The properties of the 4-form $\sigma^T$ are studied in detail in \cite{AFF} where it is shown that $\sigma^T$ measures the `degeneracy' of the 3-form $T$.

The exterior derivative $dT$ has the following  expression (see e.g. \cite{I,IP2,FI})
\begin{equation}\label{dh}
\begin{split}
dT(X,Y,Z,V)=d^{\sb}T(X,Y,Z,V) +2\sigma^T(X,Y,Z,V), \quad where\\
d^{\sb}T(X,Y,Z,V)=(\nabla_XT)(Y,Z,V)+(\nabla_YT)(Z,X,V)+(\nabla_ZT)(X,Y,V)-(\nabla_VT)(X,Y,Z).
 \end{split}
 \end{equation}
 For the curvature of $\sb$ we use the convention $ R(X,Y)Z=[\nabla_X,\nabla_Y]Z -
 \nabla_{[X,Y]}Z$ and $ R(X,Y,Z,V)=g(R(X,Y)Z,V)$. It has the well known properties
 \begin{equation}\label{r1}
 R(X,Y,Z,V)=-R(Y,X,Z,V)=-R(X,Y,V,Z).
 \end{equation}
  The first Bianchi identity for $\nabla$ can be written in the form (see e.g. \cite{I,IP2,FI})
 \begin{equation}\label{1bi}
 \begin{split}
 R(X,Y,Z,V)+ R(Y,Z,X,V)+ R(Z,X,Y,V)\\ 
 =dT(X,Y,Z,V)-\sigma^T(X,Y,Z,V)+(\nabla_VT)(X,Y,Z).
 \end{split}
 \end{equation}
It is proved in \cite[p. 307]{FI} that the curvature of  a metric connection $\sb$ with totally skew-symmetric torsion $T$  satisfies also the  identity
 \begin{equation}\label{gen}
 \begin{split}
 R(X,Y,Z,V)+ R(Y,Z,X,V)+ R(Z,X,Y,V)-R(V,X,Y,Z)-R(V,Y,Z,X)-R(V,Z,X,Y)\\
 =\frac32dT(X,Y,Z,V)-\sigma^T(X,Y,Z,V).
 \end{split}
 \end{equation}
 One gets from \eqref{gen} and \eqref{1bi} that the curvature of the torsion connection satisfies the identity 
 \begin{equation}\label{1bi1}
 \begin{split}
R(V,X,Y,Z)+R(V,Y,Z,X)+R(V,Z,X,Y)= -\frac12dT(X,Y,Z,V)+(\nabla_VT)(X,Y,Z)
 \end{split}
 \end{equation}
 \begin{dfn} We say that the curvature $R$ satisfies the Riemannian first Bianchi identity if 
\begin{equation}\label{RB}
R(X,Y,Z,V)+R(Y,Z,X,V)+R(Z,X,Y,V)=0.
\end{equation}
\end{dfn}
 It is well known algebraic fact that \eqref{r1} and \eqref{RB} imply $R\in S^2\Lambda^2$, i.e 
 \begin{equation}\label{r4}
 R(X,Y,Z,V)=R(Z,V,X,Y).
 \end{equation}
 Note that, in general, \eqref{r1} and \eqref{r4} do not imply \eqref{RB}.

We know from  \cite[Corollary~3.4]{I} that a metric connection $\sb$ with totally skew-symmetric torsion $T$ satisfies 
\eqref{r4}  if and only if  $\sb T$ is a 4-form.  
 We  need the following  result from \cite{IS}
\begin{thrm} \cite[Theorem~1.2]{IS}\label{tFBI}
A metric connection $\sb$ with torsion 3-form $T$ satisfies the Riemannian first Bianchi identity exactly when the following identities hold 
\begin{equation*}
dT=-2\nabla T=\frac23\sigma^T.
\end{equation*}
In this case, the torsion  $T$ is parallel with respect to a metric connection with torsion 3-form $\frac13T$ \cite{AF} and therefore has a constant norm, $||T||^2=const.$
\end{thrm}
\noindent The   Ricci tensors and scalar curvatures of  $\LC$ and $\sb$ are related by (\cite[Section~2]{FI}, \cite [Prop. 3.18]{GFS})
\begin{equation}\label{rics}
\begin{split}
Ric^g(X,Y)=Ric(X,Y)+\frac12 (\delta T)(X,Y)+\frac14\sum_{i,j=1}^nT(X,e_i,e_j)T(Y,e_i,e_j);\\
Scal^g=Scal+\frac14||T||^2,\qquad Ric(X,Y)-Ric(Y,X)=-(\delta T)(X,Y),
\end{split}
\end{equation}
where $\delta=(-1)^{np+n+1}*d*$ is the co-differential acting on $p$-forms and $*$ is the Hodge star operator satisfying $*^2=(-1)^{p(n-p)}$.

Following \cite{GFS}, we denote 
$T^2_{ij}=T_{iab}T_{jab}:=\sum_{a,b=1}^nT_{iab}T_{jab}.$
Then the first equality in \eqref{rics} reads
$$Ric^g=Ric+\frac12\delta T+\frac14T^2.$$


\section{$G_2$ structure}

We recall some notions of $G_2$ geometry. Endow $\mathbb R^7$ with its standard orientation and
inner product. Let $\{e_1,\dots,e_7\}$ be an oriented orthonormal basis which we identify with
the dual basis via the inner product.Write $e_{i_1 i_2\dots i_p}$ for the
monomial $e_{i_1} \wedge e_{i_2} \wedge \dots \wedge e_{i_p}$
We shall omit the $\sum$-sign understanding summation on any pair of equal indices.

Consider the three-form $\p$ on $\mathbb R^7$ given by
\begin{equation}\label{g2}
\begin{split}
\p=e_{127}+e_{135}-e_{146}-e_{236}-e_{245}+e_{347}+e_{567}.
\end{split}
\end{equation}
The subgroup of $GL(7,\mathbb R)$ fixing $\p$ is the exceptional Lie group $G_2$. It is a compact,
connected, simply-connected, simple Lie subgroup of $SO(7)$ of dimension 14 \cite{Br}.  The Lie algebra is denoted by 
 $\frak{g}_2$ and it is isomorphic to the 2-forms satisfying 7 linear equations, namely
 $\frak{g}_2\cong  \{\alpha\in \Lambda^2(M) \vert
  *(\alpha\wedge\p) =- \alpha\}$
The 3-form
$\p$ corresponds to a real spinor $\epsilon$ and therefore,
$G_2$ can be identified as the isotropy group of a non-trivial
real spinor.

The Hodge star operator supplies the 4-form $\ph=*\p$ given by
\begin{equation*}
\begin{split}
\ph=*\p=e_{1234}+e_{3456}+e_{1256}-e_{2467}+e_{1367}+e_{2357}+e_{1457}
\end{split}
\end{equation*}
We recall that in dimension seven, the Hodge star operator satisfies $*^2=1$  and has the properties
\begin{equation}\label{star}
\begin{split}
*(\alpha\wedge\gamma)=(-1)^k\alpha\lrcorner *\gamma, \qquad \alpha\in \Lambda^1, \quad \gamma\in \Lambda^k;\\
*(\beta\wedge\p)=\beta\lrcorner *\p, \quad \beta\in \Lambda^2\qquad 
*(\beta\wedge *\p)=\beta\lrcorner \p, \quad \beta\in \Lambda^2.
\end{split}
\end{equation}

We let the expressions 
$$
 \p=\frac16\p_{ijk}e_{ijk}, \quad  \ph=\frac1{24}\ph_{ijkl}e_{ijkl}
$$
and have the  identites (c.f. \cite{Br1,Kar1,Kar})
\begin{equation}\label{iden}
\begin{split}
\p_{ijk}\p_{ajk}=6\delta_{ia};\qquad \p_{ijk}\p_{ijk}=42;\\
 \p_{ijk}\p_{abk}=\delta_{ia}\delta_{jb}-\delta_{ib}\delta_{ja}+\ph_{ijab};\quad
\p_{ijk}\ph_{abjk}=4\p_{iab};\\
\p_{ijk}\ph_{kabc}=\delta_{ia}\p_{jbc}+\delta_{ib}\p_{ajc}+\delta_{ic}\p_{abj}-\delta_{aj}\p_{ibc}-\delta_{bj}\p_{aic}-\delta_{cj}\p_{abi};\\
\ph_{ijkl}\ph_{abkl}=4\delta_{ia}\delta_{jb}-4\delta_{ib}\delta_{ja}+2\ph_{ijab};\quad \ph_{ijkl}\ph_{ajkl}=24\delta_{ia};\\
3\ph_{ijkl}\ph_{abcl}=3\big(\delta_{ia}\delta_{jb}\delta_{kc}+\delta_{ib}\delta_{jc}\delta_{ka}+\delta_{ic}\delta_{ja}\delta_{kb}
-\delta_{ia}\delta_{jc}\delta_{kb}-\delta_{ib}\delta_{ja}\delta_{kc}-\delta_{ic}\delta_{jb}\delta_{ka}\big)\\
-\big(\p_{ajk}\p_{ibc}+\p_{bjk}\p_{ica}+\p_{cjk}\p_{iab}\big)
-\big(\p_{iak}\p_{jbc}+\p_{ibk}\p_{jca}+\p_{ick}\p_{jab}  \big)\\
-\big(\p_{ija}\p_{kbc}+\p_{ijb}\p_{kca}+\p_{ijc}\p_{kab}  \big)
+\big(\delta_{ia}\ph_{jkbc}+\delta_{ib}\ph_{jkca}+\delta_{ic}\ph_{jkab}  \big)\\
+\big( \delta_{ja}\ph_{kibc}+\delta_{jb}\ph_{kica}+\delta_{jc}\ph_{kiab} \big)
+\big( \delta_{ka}\ph_{ijbc}+\delta_{kb}\ph_{ijca}+\delta_{kc}\ph_{ijab} \big).
\end{split}
\end{equation}
Note that in \cite{Kar} a different sign convention is used so the identities have
some signs different (see Remark~\ref{remconv} below).

A $G_2$ structure on a 7-manifold $M$ is a reduction of the structure group
of the tangent bundle to the exceptional Lie group $G_2$. Equivalently, there exists a nowhere
vanishing differential three-form $\p$ on $M$ and local frames of the cotangent bundle with
respect to which $\p$ takes the form \eqref{g2}. The three-form $\p$ is called the fundamental
form of the $G_2$ manifold $M$ \cite{Bo}.
We will say that the pair $(M, \p)$ is a $G_2$ manifold with $G_2$ structure (determined
by) $\p$.  Alternatively, a $G_2$ structure can be described by the existence of a two-fold
vector cross product  on the tangent spaces of $M$ (see e.g. \cite{Gr}).

It is well known that the fundamental form of a $G_2$ manifold determines a Riemannian metric which  is referred to as the metric induced by $\p$. We write $\LC$ for the associated Levi-Civita
connection. 


The action of $G_2$  on the tangent space induces an
action of $G_2$ on $\Lambda^k(M)$ splitting the exterior algebra into orthogonal subspaces, where
$\Lambda^k_l$ corresponds to an $l$-dimensional $G_2$-irreducible subspace of $\Lambda^k$:
\begin{equation*}\label{dec}
\begin{split}
\Lambda^1(M)=\Lambda^1_7,  \quad \Lambda^2(M)=\Lambda^2_7\oplus\Lambda^2_{14}, \qquad \Lambda^3(M)=\Lambda^3_1\oplus\Lambda^3_7\oplus\Lambda^3_{27},
\end{split}
\end{equation*}
where
\begin{equation}\label{dec2}
\begin{split}
\Lambda^2_7=\{\phi\in \Lambda^2(M) | *(\phi\wedge\p)=2\phi\};\\
\Lambda^2_{14}=\{\phi\in \Lambda^2(M) | *(\phi\wedge\p)=-\phi\}\cong g_2;\\
\Lambda^3_1=t\p,\quad t\in \mathbb R;\\
\Lambda^3_7=\{*(\alpha\wedge\p) |  \alpha\in\Lambda^1\}=\{\alpha\lrcorner\ph\};\\
\Lambda^3_{27}=\{\gamma\in \Lambda^3(M) | \gamma\wedge\p=\gamma\wedge\ph=0\}.
\end{split}
\end{equation}
Denote by $S^2_-$ the space of symmetric traceless 2-tensors, 
\begin{equation}\label{2s}
h(X,Y)=h(Y,X),\qquad tr_gh=0.
\end{equation}
It is known (see \cite{Br1,Kar,Kar1}) that the map $\gamma: S^2_-\Leftrightarrow \Lambda^3_{27}$ defined by 
\begin{equation}\label{s12}
\gamma(h_{ij})=h_{ip}\p_{pjk}+h_{jp}\p_{pki}+h_{kp}\p_{pij},\qquad  h_{im}=\gamma^{-1}(B_{ijk})=\frac14B_{ijk}\p_{mjk}
\end{equation}
is an isomorphism of $G_2$ representations.

We recall and give a proof of the next algebraic fact stated  in the proof of \cite[Theorem~5.4]{FI}.
\begin{prop}\cite[p. 319]{FI}\label{4-for}
Let $A$ be a 4-form and define the 3-forms $B_X=(X\lrcorner A)$ for any $X\in T_pM$. If  the 3-forms $B_X\in\Lambda^3_{27}$ then the four form $A$ vanishes identically, $A=0$
\end{prop}
\begin{proof}
According to \eqref{s12} and \eqref{2s} the tensor
\begin{multline}\label{c1}
C(X,Y,Z)=4\gamma^{-1}(X\lrcorner A)(Y,Z)=A(X,Y,e_i,e_j)\p(Z,e_i,e_j)\\=C(X,Z,Y)=-C(Z,X,Y)=-C(Z,Y,X)=C(Y,Z,X)=C(Y,X,Z)=-C(X,Y,Z)
\end{multline}
and therefore vanishes. We  obtain from \eqref{c1} using  \eqref{iden}
\begin{equation}\label{ll1}
0=A_{pbij}\p_{ijs}\p_{kls}=A_{pbij}\big[\ph_{ijkl}+\delta_{ik}\delta_{jl}-\delta_{il}\delta_{jk}  \big]=A_{pbij}\ph_{ijkl}+2A_{pbkl}.
\end{equation}
The identity \eqref{ll1} together with \eqref{iden} yields
\begin{equation}\label{g222}
-2A_{pbkl}\p_{kls}=A_{pbij}\ph_{ijkl}\p_{kls}=4A_{pbij}\p_{ijs} \Longrightarrow A_{pbij}\p_{ijs}=0.
\end{equation}
The equalities \eqref{c1}, \eqref{ll1}, \eqref{g222} and \eqref{iden} imply
\begin{equation*}
4A_{pkac}=-2A_{pbkl}\ph_{abcl}=A_{pbij}\ph_{ijkl}\ph_{abcl}=
-4A_{pkac}
\end{equation*}
Hence $A=0$ which completes the proof of Proposition~\ref{4-for}.
\end{proof}

\begin{rmrk}\label{remconv}There is another different orientation convention for $G_2$ structures.
In the other convention, the eigenvalues of the operator $\beta\rightarrow *(\beta\wedge\p)$ are -2 and +1 instead of +2 and -1, respectively. 
\end{rmrk}

In~\cite{FG}, Fernandez and Gray divide $G_2$ manifolds into 16 classes
according to how the covariant derivative $\LC\p$ 
behaves with respect to its decomposition into $G_2$ irreducible components
(see also~\cite{CSal,GKMW,Br1}).  If the fundamental form is parallel with respect to
the Levi-Civita connection, $\LC\p=0$,
 then the Riemannian holonomy group is contained
in $G_2$. In this case the induced metric on the
$G_2$ manifold is Ricci-flat, a fact first observed by Bonan~\cite{Bo}.  It
was also shown in \cite{FG} that a $G_2$ manifold
is parallel precisely when the fundamental form is harmonic, i.e. $d\p=d*\p=0$.
The first examples of complete parallel $G_2$ manifolds were constructed by Bryant and
Salamon~\cite{BS,Gibb}.  Compact examples of parallel $G_2$ manifolds were
obtained first by Joyce~\cite{J1,J2,J3} and  with another construction by Kovalev~\cite{Kov}.

The Lee form $\theta$ is defined by \cite{Cabr} (see also \cite{Br})
\begin{equation}\label{g2li}
\theta=-\frac{1}{3}*(* d\p\wedge\p)=\frac13*(*d*\p\wedge *\p)=-\frac13*(\delta\p\wedge*\p)=-\frac13\delta\p\lrcorner\p,
\end{equation}
where $\delta=(-1)^k*d*$ is the codifferential  acting on $k$-forms  and one applies \eqref{star} to get the last identity.

The failure of the holonomy group of the Levi-Civita connection $\sb^g$ of the metric $g$ to reduce to $G_2$
can also be measured by the intrinsic torsion $\tau$, which is identified with $d\p$ an $d*\p=d\ph$, 
and can be decomposed into four basic classes \cite{CSal,Br1}, $\tau \in W_1\oplus W_7 \oplus W_{14}\oplus W_{27}$ which gives another description of the Fernandez-Gray classification \cite{FG}.  We list below those of them which
we will use later.

\noindent - $\tau \in W_1$. The class of nearly parallel  (weak holonomy) $G_2$ manifold defined by
$d\p=const.*\p, \quad d*\p=0$.

\noindent - $\tau \in W_7$. The class of locally conformally parallel $G_2$ spaces characterized by
$d*\p=\theta\wedge *\p, \quad d\p=\frac34\theta\wedge\p$.
 - $\tau \in W_{27}$. The class of pure integrable  $G_2$ manifolds determined by  $d\p\wedge\p=0$ and $d*\p=0$.

\noindent - $\tau \in W_1\oplus W_{27}$. The class of cocalibrated $G_2$ manifold, 
determined by the condition $d*\p=0$. 
 
\noindent - $\tau \in W_1\oplus W_7\oplus W_{27}$. The class of integrable $G_2$ manifold determined by the condition $d*\p=\theta\wedge *\p$.  An analog of the Dolbeault cohomology is investigated in \cite{FUg}.  In this class, the exterior derivative of the Lee form lies in the Lie algebra $\frak{g}_2$, $d\theta\in \Lambda^2_{14}$ \cite{Kar1}. 
This is the  class which we are interested in.

 \noindent- $\tau \in W_7\oplus W_{27}$. This class is determined by the conditions $d\p\wedge\p=0$ and $d*\p=\theta\wedge*\p$ and is of great interest in supersymmetric heterotic string theories in dimension seven  \cite{GKMW,FI,FI1,GMW,GMPW,OLS}. We call this class \emph{strictly}  integrable   $G_2$ manifolds .

An important sub-class of the integrable $G_2$ manifolds is determined in the next
\begin{dfn}\label{ctype}
An integrable $G_2$ structure is said to be of constant type if the function $(d\p,*\p)=const.$.
\end{dfn}
For example, the nearly parallel as well as the strictly integrable $G_2$ manifolds are integrable of constant type. The integrable $G_2$ manifolds of constant type appear also in  the $G_2$ heterotic supergravity where  the constant  $(d\p,*\p)$ is interpreted as the AdS radius \cite{Oss1,Oss2} see also  \cite[Section~5.2.1]{AMP}.

If the Lee form of an integrable $G_2$ structure  vanishes, $\theta=0$ then the $G_2$ structure is
co-calibrated. If the Lee form of an integrable $G_2$ structure is closed,
$d\theta=0$ then the $G_2$ structure is locally conformally
equivalent to a co-calibrated one \cite{FI1} (see also \cite{Kar1}) and if the Lee form is an exact form then it is (globally) conformal to a co-calibrated one.  
It is known due to  \cite[Theorem~3.1]{FI1} that for any  integrable $G_2$ structure  on a compact manifold, there exists a unique integrable $G_2$ structure conformal to the original one with co-closed Lee form,  called \emph{the Gauduchon $G_2$ structure}.

We recall the following
\begin{dfn}
The curvature $R$ of a linear connection on a $G_2$ manifold  is a $G_2$-instanton if  the curvature 2-form lies in the Lie algebra $\frak{g}_2\cong\Lambda^2_{14}$. This is equivalent to  the  identities:
\begin{equation}\label{rri}
R_{abij}\p_{abk}=0 \Longleftrightarrow R_{abij}\ph_{abkl}=-2R_{klij}.
\end{equation}
\end{dfn}
\section{The $G_2$-connection with skew-symmetric torsion}
The necessary and sufficient  conditions for  a 7-dimensional manifold with a $G_2$ structure to admit a metric connection with torsion 3-form preserving the $G_2$ structure are found in \cite{FI} ( see also \cite{GKMW,FI1,GMW,GMPW}). 
\begin{thrm}\cite[Theorem~4.8]{FI}\label{cythm1}
Let $(M,\p)$ be a  smooth  manifold with a $G_2$ structure $\p$.

The next two conditions are equivalent
\begin{enumerate}
\item[a)] The $G_2$ structure $\p$ is integrable,
\begin{equation}\label{cycon}
\begin{split}
d*\p=\theta\wedge *\p.
\end{split}
\end{equation}
\item[b)] There exists a unique $G_2$-connection $\sb$ with torsion 3-form preserving the $G_2$ structure,

$
\sb g=\sb\p=\sb\ph=0.
$
The torsion of $\sb$ is given by
\begin{equation}\label{torcy}
\begin{split}
T=-*d\p +*(\theta\wedge\p)+\frac{1}{6}(d\p,*\p)\p.
\end{split}
\end{equation}
\end{enumerate}
\end{thrm}
The unique linear connection $\sb$  preserving the $G_2$ structure with totally skew-symmetric torsion is called \emph{the characteristic connection}.  The curvature and the Ricci tensor of $\sb$  will be called \emph{characteristic curvature} and \emph{characteristic Ricci tensor}, respectively.

If the $G_2$ structure is nearly parallel then  the torsion is parallel with respect to the characteristic connection, $\sb T=0$  \cite{FI}.

We recall that the $G_2$-Hull connection $\sb^h$ is defined to be the metric connection with torsion $-T$, where $T$ is the torsion of the characteristic connection,
\begin{equation}\label{hu}\sb^h=\LC-\frac12T=\sb-T.
\end{equation}

\subsection{The torsion and the Ricci tensor of the characteristic connection}
We obtain from \eqref{torcy} using \eqref{star} that
\begin{equation}\label{torcy1}
\begin{split}
T=-*d\p +*(\theta\wedge\p)+\frac{1}{6}(d\p,\ph)\p 
=-*d*\Phi-\theta\lrcorner\Phi+\frac{1}{6}(d\p,\ph)\p
=-\delta\Phi-\theta\lrcorner\Phi+\frac{1}{6}(d\p,\ph)\p.
\end{split}
\end{equation}
Write $\delta\ph$  in terms $\LC$  and then in terms of $\sb$ using \eqref{tsym} and  $\sb\Phi=0$ to get 
\begin{equation}\label{dphi}
\begin{split}
-\delta\Phi_{klm}=\sb^g_j\Phi_{jklm}=\sb_j\Phi_{jklm}-\frac12T_{jsk}\Phi_{jslm}+\frac12T_{jsl}\Phi_{jskm}-\frac12T_{jsm}\Phi_{jskl}\\
=-\frac12T_{jsk}\Phi_{jslm}+\frac12T_{jsl}\Phi_{jskm}-\frac12T_{jsm}\Phi_{jskl}.
\end{split}
\end{equation}
Substituting \eqref{dphi} into \eqref{torcy1}, we obtain  the following formula of the 3-form torsion $T$,
\begin{equation}\label{torcy2}
T_{klm}=-\frac12T_{jsk}\Phi_{jslm}+\frac12T_{jsl}\Phi_{jskm}-\frac12T_{jsm}\Phi_{jskl}-\theta_s\Phi_{sklm}+\lambda\p_{klm}, 
\end{equation}
where the function $\lambda$ is defined by the  scalar product
\begin{equation}\label{lm}
\lambda =\frac16(d\p,\ph)=\frac1{42}d\p_{ijkl}\ph_{ijkl}=\frac1{36}\delta\ph_{klm}\p_{klm}.
\end{equation}
Applying \eqref{iden},  it is easy to check from \eqref{g2li},  \eqref{torcy2} and \eqref{dphi} that the Lee form $\theta$ can be written in the form 
\begin{equation}\label{tit}
\begin{split}
\theta_i=\frac16T_{jkl}\Phi_{jkli}=\frac16T_{jkl}\p_{jks}\p_{lis}=\frac1{18}\delta\ph_{jkl}\ph_{jkli}.
\end{split}
\end{equation}
We calculate from \eqref{torcy2} the function $\lambda$  in terms of the torsion $T$ using \eqref{iden} as follows
\begin{equation}\label{lmt}
T_{klm}\p_{klm}=-\frac32T_{jsk}\Phi_{jslm}\p_{klm}+42\lambda=-6T_{jsk}\p_{jsk}+42\lambda \Longrightarrow \lambda=\frac16T_{klm}\p_{klm}
\end{equation}
Similarly, we obtain the next identities
\begin{equation}\label{lmt0}
\begin{split}
T_{kli}\p_{klj}-T_{klj}\p_{kli}=-2\theta_s\p_{sij}.\\
\sigma^T_{iabc}\p_{abc}=-3T_{abs}\p_{abc}T_{sci}=3\theta_s\p_{skt}T_{kti},
\end{split}
\end{equation}
where $\sigma^T$ is defined in \eqref{sigma}.

For the $\Lambda^3_{27}$ component $(\delta\ph)^3_{27}$ of $\delta\ph$ we get  taking into account  \eqref{dphi}, \eqref{iden}, \eqref{lm} and \eqref{tit} that 
\begin{equation}\label{327}
(\delta\ph)^3_{27}=\delta\ph+\frac34\theta\lrcorner\ph-\frac67\lambda\p
\end{equation}
because we calculated $(\delta\ph)^1_{27}=\frac67\lambda\p$ and $(\delta\ph)^7_{27}=-\frac34\theta\lrcorner\ph$.

The equalities \eqref{327} and \eqref{torcy1} yield the next formulas for the 3-form torsion $T$  and its norm $||T||^2$, \cite{FI},
\begin{equation}\label{torcy3}
\begin{split}
T=-(\delta\ph)^3_{27}-\frac14\theta\lrcorner\ph+\frac17\lambda\p=-(\delta\ph)^3_{27}-\frac14\theta\lrcorner\ph+\frac1{42}(d\p,\ph)\p,
\\ ||T||^2= 
||(\delta\ph)^3_{27}||^2+\frac{3}2||\theta||^2+\frac1{42}(d\p,\ph)^2=||\delta\ph||^2-12||\theta||^2-\frac56(d\p,*\p)^2.
\end{split}
\end{equation}
\subsection{The characteristic Ricci tensor}
Studying the properties of the characteristic Ricci tensor, we have
\begin{thrm}\label{maing2}
On an integrable     $G_2$ manifold $(M,\p)$ the next conditions are equivalent:
\begin{itemize}
\item[a)] The characteristic Ricci tensor  is symmetric, $Ric(X,Y)=Ric(Y,X)$;
\item[b)] The two form $d^{\sb}\theta(X,Y)=(\sb_X\theta)Y-(\sb_Y\theta)X$ belongs to $\Lambda^2_7$ and it is given by
 \begin{equation}\label{newsymth}
d^{\sb}\theta=d\lambda\lrcorner\p, \quad \Leftrightarrow 
 \sb_X(d\p,*\p)=d^{\sb}\theta(e_a,e_b)\p(X,e_a,e_b)=6\sb_X\lambda.
 \end{equation}
 \end{itemize}
\end{thrm}
\begin{proof}
We calculate from \eqref{torcy} using \eqref{star}, \eqref{cycon},  \eqref{dec2} and the fact obseved in \cite{Kar1} that $d\theta\in\Lambda^2_{14}$ 
\begin{multline}\label{rrr}
-\delta T=*d*T=*(d\theta\wedge\p)-*(\theta\wedge d\p)+*(d\lambda\wedge\ph) +*(\lambda\theta\wedge\ph)\\=-d\theta-*(\theta\wedge d\p)+*[d\lambda+\lambda\theta)\wedge\ph]=-d\theta-\theta\lrcorner*d\p+(d\lambda+\lambda\theta)\lrcorner\p\\=-d\theta-\theta\lrcorner\delta\ph+(d\lambda+\lambda\theta)\lrcorner\p=-d\theta+\theta\lrcorner T-\lambda\theta\lrcorner\p+(d\lambda+\lambda\theta)\lrcorner\p=-d\theta+\theta\lrcorner T+d\lambda\lrcorner\p, 
\end{multline}
where we have applied \eqref{torcy1} in the third line.

On the other hand, \eqref{tsym} yields
\begin{equation}\label{ttg2}
d\theta=d^{\sb}\theta+\theta\lrcorner T,
\end{equation}
 which substituted into \eqref{rrr} gives
\begin{equation}\label{deltaT}
\delta T=d^{\sb}\theta-d\lambda\lrcorner\p.
\end{equation}
Apply \eqref{rics} to \eqref{deltaT} to achieve the equivalence of a) and b). The proof is completed.
\end{proof}
We  obtain from \eqref{ttg2}, \eqref{deltaT} and $d\theta\in\Lambda^2_{14}$ that
\begin{equation}\label{nnew1}
 -\delta T_{bc}\p_{bci}=\sb_sT_{sbc}\p_{bci}=\theta_sT_{skt}\p_{kti}+6d\lambda_i.
\end{equation}
As straightforward consequences of Theorem~\ref{maing2}, we have
\begin{prop}\label{symsonst}
An integrable $G_2$ manifold with symmetric characteristic Ricci tensor   is of constant type, $6\lambda=(d\p,*\p)=const$, if and only if $\sb\theta$ is symmetric, $d^{\sb}\theta=0$.

A co-calibrated  $G_2$ structure, $d*\p=0$, has symmetric characteristic Rici tensor  if and only if the co-calibrated $G_2$ structure is of constant type, $6\lambda=(d\p,*\p)=const.$

A strictly integrable $G_2$ structure, $d*\p=\theta\wedge*\p, \quad d\p\wedge\p=0$ has symmetric characteristic Ricci tensor if and only if $\sb\theta$ is symmetric.
\end{prop}
Note that the second statement in the above proposition is observed   in \cite{FMR}.

\noindent On a locally conformally parallel $G_2$ manifold, $\tau \in W_7$,  \eqref{torcy} reads $T=\frac13*d\p$ yielding to $\delta T=0$ and
\begin{cor}
The characteristic Ricci tensor of a locally conformally parallel $G_2$ manifold is symmetric.
\end{cor}
The structure of compact locally conformally parallel $G_2$ manifolds is described in \cite{IPP}.

Explicit formulas of the characteristic Ricci tensor of an integrable $G_2$ manifold are presented in \cite{FI,FI1}. Below, we give another proof of this fact for completeness. We have
\begin{thrm}\cite{FI, FI1}\label{thRic}
The characteristic Ricci tensor $Ric$ and its scalar curvature $Scal$  are given by
\begin{equation}\label{ricg2}
\begin{split}
Ric_{ij}=\frac1{12}dT_{iabc}\ph_{jabc}-\sb_i\theta_j;\\
Scal=3\delta\theta+2||\theta||^2-\frac13||T||^2+\frac1{18}(d\p,\ph)^2
=3\delta\theta+6||\theta||^2-\frac13||\delta\ph||^2+\frac13(d\p,\ph)^2.
\end{split}
\end{equation}
The Riemannian scalar curvature $Scal^g$ of an integrable $G_2$ manifold is given by
\begin{equation}\label{scal1}
\begin{split}
Scal^g=3\delta\theta+2||\theta||^2-\frac1{12}||T||^2+\frac1{18}(d\p,\ph)^2
=3\delta\theta+3||\theta||^2-\frac1{12}||\delta\ph||^2+\frac5{36}(d\p,\ph)^2;
\end{split}
\end{equation}
The next equivalent identities hold
\begin{equation}\label{dtt}
\begin{split}
dT_{iabc}\p_{abc}+2\sb_iT_{abc}\p_{abc}=
dT_{iabc}\p_{abc}+12d\lambda_i=0;\\
 3\sb_aT_{bci}\p_{abc}-2\sigma^T_{iabc}\p_{abc}-18d\lambda_i=0;\\
\sb_aT_{bci}\p_{abc}-2\theta_sT_{skt}\p_{kti}-6d\lambda_i=0.
\end{split}
\end{equation}
\end{thrm}
\begin{proof}
 Since $\sb\p=\sb\ph=0$ the curvature of the characteristic connection lies in the Lie algebra $\frak{g}_2$, i.e. 
\begin{equation}\label{rr}
\begin{split}
R(X,Y,e_i,e_j)\p(e_i,e_j,Z)=0 \Longleftrightarrow R(X,Y,e_i,e_j)\ph(e_i,e_j,Z,V)=-2R(X,Y,Z,V).\\
R_{ijab}\p_{abk}=0 \Longleftrightarrow R_{ijab}\ph_{abkl}=-2R_{ijkl}.
\end{split}
\end{equation}
We have from \eqref{rr} using \eqref{1bi1}, \eqref{tit} and \eqref{dh}  that the Ricci tensor $Ric$ of  $\sb$ is given by
\begin{multline}\label{ricdt}
2Ric_{ij}=R_{iabc}\ph_{jabc}=\frac13\big[R_{iabc}+R_{ibca}+R_{icab} \big]\ph_{jabc}=\frac16dT_{iabc}\ph_{jabc}+\frac13\sb_iT_{abc}\ph_{jabc}.
\end{multline}
Apply \eqref{tit} to complete the proof of the first identity in \eqref{ricg2}.
Similarly, we have
\[0=R_{iabc}\p_{abc}=\frac13\big[R_{iabc}+R_{ibca}+R_{icab} \big]\p_{abc}=\frac16dT_{iabc}\p_{abc}+\frac13\sb_iT_{abc}\p_{abc}\]
which proves the first equality in \eqref{dtt}. Apply \eqref{dh} to achieve the second and \eqref{lmt0} to get the third.

We obtain  from \eqref{torcy2} using \eqref{iden} 
\begin{equation}\label{ng4}
\begin{split}
\sigma^T_{jabc}\ph_{jabc}=3T_{jas}T_{bcs}\ph_{jabc}=-2||T||^2+12||\theta||^2+12\lambda^2
\end{split}
\end{equation}
We calculate from \eqref{dh} applying \eqref{tit}, \eqref{ng4} and \eqref{torcy3}
\begin{equation}\label{g22}
\begin{split}
dT_{jabc}\ph_{jabc}=4\sb_jT_{abc}\ph_{jabc}+2\sigma^T_{jabc}\ph_{jabc}=-24\sb_j\theta_j-4||T||^2+24||\theta||^2+24\lambda^2\\
=24\delta\theta -4||\delta\ph||^2+72||\theta||^2+4(d\p,\ph)^2
\end{split}
\end{equation}
Take the trace in the first identity in \eqref{ricg2} substitute  \eqref{g22} into the obtained equality and use \eqref{lm} and \eqref{torcy3} to get the second identity in \eqref{ricg2}. 

The equality  \eqref{scal1} follows from \eqref{rics}, the second identity in \eqref{ricg2} and \eqref{torcy3}.
\end{proof}
\begin{rmrk}
The Riemannian Ricci tensor and the Riemannian scalar curvature of a general $G_2$ manifold are calculated in \cite{Br1}.
\end{rmrk}
We obtain from the proof the next corollary, first established by Bonan \cite{Bo} for a parallel $G_2$ spaces.
\begin{cor}\label{corfb}
If the curvature of the characteristic connection satisfies the Riemannian first Bianchi identity then the Ricci tensor vanishes.
\end{cor}
We get from \eqref{g22} the next 
\begin{cor}
On a co-calibrated $G_2$ structure with closed torsion one has $||\delta\ph||^2=36\lambda^2=(d\p,\ph)^2$. 

In particular, a strictly  integrable co-calibrated $G_2$ structure with closed torsion is parallel, $\sb\p=0$. 
\end{cor}
\begin{cor}\label{thconj2}
Let $(M,\p)$ be an integrable  $G_2$-manifild with vanishing scalar curvature  of the characteristic connection, $Scal=0$. If the structure is co-calibrated then $dT\in\Lambda^4_7\oplus\Lambda^4_{27}$, $dT_{ijkl}\ph_{ijkl}=0.$

In particular,  $||T||^2=6\lambda^2$ which is equivalent to $||\delta\ph||^2=(d\p,\ph)^2.$

\end{cor}

\subsection{Compact  Gauduchon $G_2$ manifolds} 
In this subsection, we recall the notion of conformal deformations of a given $G_2$ structure $\p$  from \cite{FG,FI1,Kar1} and proof Theorem~\ref{cooo}. 

Let $\bar\p=e^{3f}\p$ be a conformal deformation of $\p$. The induced metric $\bar g=e^{2f}g$ and $\bar{*}\bar\p=e^{4f}*\p$, where $\bar{*}$ is the Hodge star operator with respect to $\bar{g}$. The class of integrable $G_2$ structures is invariant under conformal deformations. An easy calculations give $(d\bar\p,\bar{*}\bar\p)=e^{-f}(d\p,*\p)$ which compared with \eqref{lm} yields $\bar\lambda=e^{-f}\lambda$. Hence, the class of strictly integrable $G_2$ manifolds, ($\lambda=0$), is invariant under conformal deformations while the class of constant non-zero type is not conformally invariant.

The Lee forms are connected by $\bar\theta=\theta+4df$. Using the expression of the Gauduchon theorem 
 in terms of a Weyl structure \cite [Appendix 1]{tod},
one can find, in a unique way, a conformal $G_2$ structure such that the corresponding Lee 
1-form is coclosed with respect to the induced metric due to \cite[Theorem~3.1]{FI1}.

Furthermore, we establishe the following:
\begin{thrm}\label{RG}
Let $(M,\p)$ be a compact integrable $G_2$ manifold with a Gauduchon $G_2$ structure, $\delta\theta=0$.  
 If  the symmetric part of the characteristic Ricci tensor  is non-negative, $Ric(X,X)\ge 0$, then the Lee form is $\sb$-parallel and $\delta T=-d\lambda\lrcorner\p\in\Lambda^2_7$.
\end{thrm}
\begin{proof} We start with the next identity
\begin{equation}\label{iii}
\sb_i\delta T_{ij}=\frac12\delta T_{ia}T_{iaj}.
\end{equation}
 shown in \cite[Proposition~3.2]{IS} for any metric connection with a totally skew-symmetric torsion. 
We calculate the left-hand side  of \eqref{iii} applying \eqref{deltaT} as follows
\begin{equation}\label{ntss}
\begin{split}
\sb_i\delta T_{ij}=\sb_i[d^{\sb}\theta_{ij}-\sb_t\lambda\p_{tij}]=
\sb_i\sb_i\theta_j-\sb_i\sb_j\theta_i-\frac12T_{tis}\sb_s\lambda\p_{tij},
\end{split}
\end{equation}
where we applied $d^2\lambda=0$ and \eqref{tsym} to get the last term.

Substitute \eqref{ntss} into \eqref{iii} using \eqref{deltaT} to get
\begin{equation}\label{ntss1}
\sb_i\sb_i\theta_j-\sb_i\sb_j\theta_i-\frac12T_{abs}\sb_s\lambda\p_{sab}=\frac12d^{\sb}\theta_{ab}T_{abj}-\frac12T_{abj}\sb_s\lambda\p_{sab}.
\end{equation}
The Ricci identity  
$\sb_i\sb_j\theta_i=
\sb_j\sb_i\theta_i+Ric_{js}\theta_s-\frac12d^{\sb}\theta_{ai}T_{aij}$ substituted into \eqref{ntss1} yields
\begin{equation}\label{gafin}
\sb_i\sb_i\theta_j+\sb_j\delta\theta-Ric_{js}\theta_s=\frac12\sb_s\lambda\Big(T_{abs}\p_{abj}-T_{abj}\p_{abs}  \Big)=-\sb_s\lambda\theta_a\p_{asj},
\end{equation}
where we use the first identity of \eqref{lmt0} to achieve the last equality. 

Multiply the both sides of \eqref{gafin} with $\theta_j$, use  $\delta\theta=0$ together with the identity 
$$\frac12\Delta||\theta|^2=-\frac12\LC_i\LC_i||\theta||^2=-\frac12\sb_i\sb_i||\theta||^2=-\theta_j\sb_i\sb_i\theta_j-||\sb\theta||^2$$ 
to get
\begin{equation}\label{gafinf}
-\frac12\Delta||\theta||^2-Ric(\theta,\theta)-||\sb\theta||^2=0.
\end{equation}
An integration of \eqref{gafinf} over the compact $M$ and  the condition $Ric(X,X)\ge 0$ gives $\sb\theta=0=Ric(\theta,\theta)$. Now,  \eqref{deltaT} completes the proof.
\end{proof}
\subsection{Proof of Theorem~\ref{cooo}.}
\begin{proof}
Observe that the symmetricity of the Ricci tensor is equivalent to $\delta T=0$  which combined with $\sb\theta=0$  and \eqref{deltaT} gives $d\lambda=0$. These two identities together with \eqref{ricg2} and \eqref{dtt} yield
\begin{equation}\label{nnt}
dT_{iabc}\Phi_{jabc}=\sb_i\theta_j=0, \qquad dT_{iabc}\p_{abc}=d\lambda_i=0.
\end{equation}
Hence,  for any vector  $X\in T_pM$ the 3-forms $(X\lrcorner dT)\in \Lambda^3_{27}$ and  therefore  the four form $dT=0$ due to Proposition~\ref{4-for}. Now \eqref{ricg2} and \eqref{g22} show $Ric=d||T||^2=0$ which completes the proof.
\end{proof}

\section{$G_2$-instanton connections. Proof of Theorem~\ref{inst}}
In this section, we prove Theorem~\ref{inst}.
\begin{proof}
We begin with
\begin{lemma}
Let $(M,\p)$ be  an integrable $G_2$ manifold and the  curvature of the characteristic  connection $\sb$   is a $G_2$-instanton. Then  $\delta T\in\Lambda^2_{14}\cong  \frak{g}_2$.
\end{lemma}
\begin{proof}
Suppose the curvature $R$ of $\sb$ is a $G_2$-instanton. 
Multiply \eqref{gen} with  $\p$ and  apply  \eqref{rri} to get
\begin{equation}\label{rri2}
0=\Big[3R_{abci}-3R_{iabc}\Big]\p_{abc}=\Big[\frac32dT_{abci}-\sigma^T_{abci}\Big]\p_{abc}
\end{equation}
We obtain from \eqref{rri2} and  \eqref{dtt} 
$$
dT_{abci}\p_{abc}=\frac23\sigma^T_{abci}\p_{abc}=2\sb_iT_{abc}\p_{abc}
$$ 
and use \eqref{lmt} to conclude 
\begin{equation}\label{rri5}
\begin{split}
\sb_iT_{abc}\p_{abc}=6d\lambda_i=\frac13\sigma^T_{abci}\p_{abc}=-\theta_s\p_{sab}T_{abi}=-\theta_sT_{sab}\p_{abi}=d^{\sb}\theta_{ab}\p_{abi}
\end{split}
\end{equation}
Substitute \eqref{rri5} into \eqref{dtt} and \eqref{nnew1} to get
\begin{equation}\label{rri6}
\begin{split}
\delta T_{ab}\p_{abi}=0 \Leftrightarrow \delta T\in\Lambda^2_{14};\quad \sb_aT_{bci}\p_{abc}=-6d\lambda_i.
\end{split}
\end{equation}
The lemma is proved.
\end{proof}
\begin{lemma}\label{sbthetain}
Let $(M,\p)$ be  an integrable $G_2$ manifold  of constant type and the  characteristic  connection   is Ricci-flat and  is a $G_2$-instanton.

Then the Lee form $\theta$ is $\sb$-parallel, $$\sb\theta=0.$$
In particular, the Lee form is co-slosed, $\delta\theta=0$.
\end{lemma}
\begin{proof}
Multiply \eqref{gen} with $\ph$ and use \eqref{rri} to get
\begin{equation}\label{rri21}
\Big[3R_{abci}-3R_{iabc}\Big]\ph_{abcj}=-6R_{cjci}+6R_{iaaj}=6Ric_{ji}+6Ric_{ij}=\Big[\frac32dT_{abci}-\sigma^T_{abci}\Big]\ph_{abcj}.
\end{equation}

We obtain from \eqref{rri21} and \eqref{ricdt} using \eqref{dh} that
\begin{equation}\label{rri4}
\begin{split}
6Ric_{ij}-6Ric_{ji}=
\Big[-\frac12dT_{abci}-2\sb_iT_{abc}+\sigma^T_{abci} \Big]\ph_{abcj}=
-\frac32\Big[\sb_aT_{bci}+\sb_iT_{abc} \Big]\ph_{abcj}.
\end{split}
\end{equation}
Suppose $Ric=0$. Then \eqref{rri2} and \eqref{rri21} yield 
\[(3dT-2\sigma^T)_{abci}\p_{abc}=(3dT-2\sigma^T)_{abci}\ph_{abcj}=0. \]
Hence,  for any $X\in T_pM$ the 3-forms  $X\lrcorner (3dT-2\sigma^T)\in \Lambda^3_{27}$ 
and Proposition~\ref{4-for} implies 
\begin{equation}\label{dtsig}3dT-2\sigma^T=0.\end{equation}
Therefore, $$T_{abc}dT_{abcj}=\frac23T_{abc}\sigma^T_{abcj}=0,$$
where we used the identity $T_{abc}\sigma^T_{abcj}=0$ observed in \cite[Proposition~3.1]{IS}.

The second Bianchi identity for a metric connection with totally skew-symmetric torsion reads \cite[Proposition~3.5]{IS}
\begin{equation}\label{e1}
d(Scal)_j-2\sb_iRic_{ji}+\frac16d||T||^2_j+\delta T_{ab}T_{abj}+\frac16T_{abc}dT_{jabc}=0.
\end{equation}
Use  \eqref{e1} for the characteristic connection to obtain $d||T||^2=0$ since all other terms in \eqref{e1} vanish.

Substitute \eqref{dh} into \eqref{ricg2},  use \eqref{tit} and \eqref{rri4}  to get
\begin{equation}\label{su1}
0=\sb_i\theta_j-\frac1{12}(\sb_iT_{abc}-3\sb_aT_{bci}+2\sigma^T_{iabc})\ph_{jabc}=3\sb_i\theta_j-\frac1{6}\sigma^T_{iabc}\ph_{jabc},
\end{equation}
where we used $$\sb_aT_{bci}\ph_{abcj}=-\sb_iT_{abc} \ph_{abcj}=6\sb_i\theta_j$$ following from  $Ric=0$ and \eqref{rri4}.

Then we have using \eqref{tit}, \eqref{su1}, \eqref{ricg2} and \eqref{sigma}
\begin{equation}\label{li}
\begin{split}
\sb_i\theta_j=\frac16\sb_iT_{abc}\ph_{abcj}=\frac1{12}dT_{iabc}\ph_{jabc}=\frac1{18}\sigma^T_{iabc}\ph_{jabc}=
\frac16T_{abs}T_{cis}\ph_{abcj}.
\end{split}
\end{equation}
We calculate from \eqref{li} applying \eqref{torcy2}
\begin{equation}\label{part}
\begin{split}
6\sb_p\theta_k=T_{jsl}T_{lmp}\ph_{jsmk}
=-T_{klm}T_{lmp}-\frac12T_{jsk}\ph_{jslm}T_{lmp}-\theta_a\ph_{aklm}T_{lmp}+\lambda\p_{klm}T_{lmp}.
\end{split}
\end{equation}
Since $\lambda=const$, \eqref{newsymth} implies $\sb\theta$ is symmetric. Then \eqref{ttg2} yields $\theta\lrcorner T=d\theta\in \Lambda^2_{14}$. i.e. 
\begin{equation}\label{tg22}
\theta_sT_{sab}\ph_{abij}=-2\theta_sT_{sij}.
\end{equation}
Multiply \eqref{part} with $\theta_p$, use \eqref{newsymth} and \eqref{tg22} to get
\begin{equation}\label{part1}
\begin{split}
3\sb_k||\theta||^2=6\theta_p\sb_k\theta_p=6\theta_p\sb_p\theta_k=T_{jsl}T_{lmp}\ph_{jsmk}\theta_p\\
=-T_{klm}T_{lmp}\theta_p-\frac12T_{jsk}\ph_{jslm}T_{lmp}\theta_p-\theta_a\ph_{aklm}T_{lmp}\theta_p+\lambda\p_{klm}T_{lmp}\theta_p\\
=-T_{klm}T_{lmp}\theta_p+T_{jsk}T_{jsp}\theta_p+2\theta_aT_{akp}\theta_p=0.
\end{split}
\end{equation}
Thus, the norm of the Lee form is constant, $||\theta||^2=const$.

Since $Scal=0$, the second identity in \eqref{ricg2} yields
\begin{equation}\label{delth}
3\delta\theta=-2||\theta||^2+\frac13||T||^2-\frac1{18}(d\p,\ph)^2.
\end{equation}
We already know that the norm of the torsion is constant, $\sb_k||T||^2=0$,  the norm of the Lie form $\theta$ is a  constant, $\sb_k||\theta||^2=0$ due to \eqref{part1} and the last term in \eqref{delth} is also constant. 
 Now, \eqref{delth}  shows that the codifferential of $\theta$ is a constant, 
\begin{equation}\label{delc}
\sb_k\delta\theta=-\sb_k\sb_i\theta_i=0.
\end{equation}
Using \eqref{newsymth}, \eqref{delc}  and the Ricci identiy for the characteristic connection $\sb$, we calculate 
\begin{multline*}
0=\frac12\sb_i\sb_i||\theta||^2=\theta_j\sb_i\sb_i\theta_j+||\sb\theta||^2=\theta_j\sb_i\sb_j\theta_i+||\sb\theta||^2\\=\theta_j\sb_j\sb_i\theta_i
-R_{ijis}\theta_s\theta_j-\theta_jT_{ijs}\sb_s\theta_i+||\sb\theta||^2=Ric_{js}\theta_j\theta_s+||\sb\theta||^2=||\sb\theta||^2,
\end{multline*}
since $Ric=0$, $\sb\theta$ is symmetric due to \eqref{newsymth} and $\lambda=const.$. This completes the proof of the Lemma.
\end{proof}
To finish the proof of Theorem~\ref{inst} we observe  from  \eqref{tit}, \eqref{lmt},  \eqref{ricg2},  \eqref{dtt}  
 and Lemma~\ref{sbthetain} the validity of the following  identities
\begin{equation}\label{nth}
\begin{split}
dT_{pjkl}\Phi_{jkli}=12\sb_p\theta_i=0;\quad
dT_{jklp}\p_{jkl}=12\sb_p\lambda=0;
\end{split}
\end{equation}
The identities \eqref{nth} show that for any $X\in T_pM$ the 3-forms  $(X\lrcorner dT)\in \Lambda^3_{27}$. Hence, the four form $dT=0$ due to Proposition~\ref{4-for} and $\sigma^T=0$ because of \eqref{dtsig}. Now, \eqref{tsym} gives $\LC T=\sb T$. 

This  completes the proof of the Theorem~\ref{inst}.
\end{proof}
Concerning the $G_2$-Hull connection, we prove the following
\begin{thrm}\label{hul}
The curvature $R^h$ of the $G_2$-Hull connection $\sb^h$ is a $G_2$ instanton if and only if the torsion is closed, $dT=0$.
\end{thrm}
\begin{proof}
We start with the general well-known formula for the curvatures of two metric connections with totally skew-symmetric torsion $T$ and $-T$, respectively, see e.g. \cite{MS}, which applied to the curvatures of the characteristic connection and the $G_2$-Hull connection reads
\begin{equation}\label{hust}
R(X,Y,Z,V)-R^h(Z,V,X,Y)=\frac12dT(X,Y,Z,V).
\end{equation}
If $dT=0$ the result was already  observed  in \cite{MS}. Indeed, in this case  the $G_2$-Hull connection is a $G_2$ instanton since $\sb\p=0$ and the holonomy group of $\sb$ is contained in the Lie algebra $\frak{g}_2$. \cite{MS}.

For the converse, \eqref{hust} yields
\begin{equation}\label{huin1}
\begin{split}
dT_{iabc}\p_{abc}=R_{iabc}\p_{abc}+R^h_{bcai}\p_{abc}=0.\\
dT_{iabc}\ph_{jabc}=R_{iabc}\ph_{jabc}+R^h_{bcai}\ph_{jabc}=-2R_{iaja}-2R^h_{jaai}=2Ric_{ij}-2Ric^h_{ji}=0,
\end{split}
\end{equation}
where $Ric^h$ is the Ricci tensor of the $G_2$-Hull connection and  the trace of \eqref{hust} gives $Ric(X,V)-Ric^h(V,X)=0.$
The identities \eqref{huin1} show that for any $X\in T_pM$ the 3-forms  $(X\lrcorner dT)\in \Lambda^3_{27}$. Hence, the four form $dT=0$ due to Proposition~\ref{4-for} 
\end{proof}

\section{Characteristic connection with curvature $R\in S^2\Lambda^2$. Proof of Theorem~\ref{mainsu3}}
We start with  
\begin{prop}\label{mainsu}
Let $(M,\p)$ be  a co-calibrated $G_2$ manifold and the   characteristic  connection $\sb$   has curvature $R \in S^2\Lambda^2$, i.e. \eqref{r4} holds.

Then the $G_2$ manifold is of constant type, the codifferential $\delta\ph$, 
the torsion $T$,  its exterior derivative $dT$ are $\sb$-parallel, and the Riemannian scalar curvature is constant, $$\sb T=\sb dT=\sb\delta\ph=0, \qquad Scal^g=-\frac1{12}||T||^2+\frac1{18}(d\p,\ph)^2=-\frac1{12}||\delta\ph||^2+\frac5{36}(d\p,\ph)^2=const.$$
\end{prop}
\begin{proof} 
We know from \cite[Corollary~3.4]{I} that the curvature of a metric connection $\sb$ with skew-symmetric torsion $T$ satisfies \eqref{r4} if and only if $\sb T$ is a 4-form.

To show that the additional condition $\theta=0$ yields $\sb T=0$ we observe that \eqref{lmt} together with  \eqref{tit} and \eqref{newsymth} imply
$
\sb_pT_{jkl}\Phi_{jkli}=6\sb_p\theta_i=0,\quad
\sb_pT_{jkl}\p_{jkl}=6\sb_p\lambda=0.
$ 
Hence, the 3-form $(X\lrcorner\sb T)\in \Lambda^3_{27}$ and the four form $\sb T=0$ due to Proposition~\ref{4-for}. Consequently, formula \eqref{sigma} shows that the four form $\sigma^T$ is also $\sb$-parallel, $\sb\sigma^T=0$ and we get from \eqref{dh} $\sb dT=0$. Finally, the formula \eqref{scal1}  completes the proof of the Proposition~\ref{mainsu}.
\end{proof}

\subsection{Proof of Theorem~\ref{mainsu3}}

To   proof Theorem~\ref{mainsu3}  we observe  that if the curvature of the characteristic connection is symmetric in exchange the first and second pairs, $R(X,Y,Z,V)=R(Z,V,X,Y)$ then it is a $G_2$ instanton because $\sb \p=0$ and \eqref{rr} holds. We apply Theorem~\ref{inst} to conclude that $dT=\delta T=\sigma^T=0$ and $\sb T=\LC T$. Note that the condition $\sb T=\LC T$ is equivalent to $\sigma^T=0$ because of  \eqref{tsym}. Now, \eqref{dh} yields $0=d^{\sb}T=4\sb T=4\LC T$  since $\sb T$ is a 4-form and  
Theorem~\ref{tFBI} shows that the Riemannian first Bianchi identity \eqref{RB} holds.

For the converse, 
the Riemannian first Bianchi identity \eqref{RB} on an integrable $G_2$ manifold implies \eqref{r4} and the vanishing of the Ricci tensor due to Corollary~\ref{corfb}. 

Finally,  the condition c) obviously implies  a) and b). 
This  completes the proof of the Theorem~\ref{mainsu3}.

\section{Integrable $G_2$ manifolds with closed torsion}
In this section we slightly improve  \cite[Theorem~5.1]{FI} and prove Theorem~\ref{closTt} and Corollary~\ref{parallel}. 
\begin{thrm}\label{closT}
Let $(M,\p)$ be  an integrable $G_2$ manifold. The following conditions are equivalent:
\begin{itemize}
\item[a).] The torsion is closed, $dT=0$.
\item[b).] The characteristic Ricci tensor is given by
\begin{equation}\label{clos1}Ric=-\sb\theta.
\end{equation}
 In particular, the integrable $G_2$ structure with closed torsion is of constant type, $d\lambda=0$.
\end{itemize}
\end{thrm}
\begin{proof}
Theorem~\ref{closT} follows from  \cite[Theorem~5.1]{FI} and the proof of  \cite[Theorem~5.4]{FI}. 

We give here a proof for completeness.

The condition $dT=0$ and  the equality \eqref{ricdt} yield $Ric=-\sb\theta$ which  implies b).

For the converse, assume b) holds. Then  \eqref{rics} yields 
$
\delta T=d^{\nabla}\theta.
$ 
  Now the equality \eqref{deltaT} shows $d\lambda=0$ and \eqref{dtt} implies $dT_{jabc}\p_{abc}=0.$ 
The equality   \eqref{clos1} combined with the first identity in  \eqref{ricg2} yields $dT_{jabc}\Phi_{iabc}=0.$
The last two equalities  show that for any $X\in T_pM$ the 3-forms $(X\lrcorner dT)\in \Lambda^3_{27}$. Hence, the four form $dT=0$ due to Proposition~\ref{4-for} which completes the equivalences of a) and b).
\end{proof}
Note that Theorem~\ref{closT}  generalizes \cite[Proposition~4.10]{Wit}.

\subsection{Proof of Theorem~\ref{closTt} and Crollary~\ref{parallel}.}
\begin{proof}  
From Theorem~\ref{closT}  we have taking into account \eqref{rics}
\begin{equation}\label{nnewt}Ric=-\sb\theta, \quad Scal =\delta \theta, \quad \delta T_{ij}=\sb_i\theta_j-\sb_j\theta_i, \quad d\lambda=0.
\end{equation}
Substitute  \eqref{nnewt} into the second Bianchi identity \eqref{e1} and use $dT=0$ to get
\begin{equation}\label{e11}
\sb_j\delta\theta+2\sb_i\sb_j\theta_i+\delta T_{ab}T_{abj}+\frac16d||T||^2_j=0.
\end{equation}
We evaluate the second term in \eqref{e11} using the Ricci identity for $\sb$ and \eqref{nnewt} as follows
\begin{equation}\label{e111}
\sb_i\sb_j\theta_i=\sb_j\sb_i\theta_i-R_{ijis}\theta_s-T_{ijs}\sb_s\theta_i=-\sb_j\delta \theta-\theta_s\sb_j\theta_s-\frac12\delta T_{si}T_{sij}.
\end{equation}
We obtain from \eqref{e111} and \eqref{e11} 
\begin{equation}\label{tt1}
-\sb_j\delta \theta-2\theta_s\sb_j\theta_s+\frac16\sb_j||T||^2=0.
\end{equation}
Another covariant derivative of \eqref{tt1} together with \eqref{nnewt} yield
\begin{equation}\label{tt2}
\Delta\delta\theta-2\theta_s\sb_j\sb_j\theta_s-2||Ric||^2-\frac16\Delta||T||^2=0.
\end{equation}
On the other hand,  \eqref{nnewt} implies 
\begin{equation}\label{is1}
\sb_i\sb_j\theta_i=\sb_i(\sb_i\theta_j-\delta T_{ij})=\sb_i\sb_i\theta_j-\sb_i\delta T_{ij}=\sb_i\sb_i\theta_j-\frac12\delta T_{ia}T_{iaj},
\end{equation}
where we used the  identity \eqref{iii}. 
The equalities \eqref{is1} and \eqref{e11} yield
\begin{equation}\label{is2}
\sb_j\delta \theta+2\sb_i\sb_i\theta_j+\frac16\sb_j||T||^2=0.
\end{equation}
Substitute the second term in \eqref{is2} into \eqref{tt2} to get
\begin{equation}\label{fmax}
\Delta\Big(\delta\theta-\frac16||T||^2 \Big)+\theta_j\sb_j\Big( \delta\theta+\frac16||T||^2\Big)=||Ric||^2\ge 0.
\end{equation}
Suppose $Ric=0$.   This combined with $dT=0$ and \eqref{ricg2} imply $\sb\theta=0$. Consequently, $0=\sb_i\theta_i=\LC_i\theta_i=-\delta\theta, \quad d||\theta||^2=0$ and  $d||T||^2=0$ either by Theorem~\ref{cooo} or \eqref{ricg2} and \eqref{nnewt}.

Assume c) holds. Since $M$ is compact and $d||T||^2=0$ we may apply  the strong maximum principle to \eqref{fmax} (see e.g. \cite{YB,GFS}) to achieve $\delta\theta=const.=0=Ric$. Conversely, the condition $\delta\theta=0$ imply $d||T||^2=0=Ric$ by the strong maximum principle applied to \eqref{fmax}. Hence a),  b) and c) are equivalent.

Finally, to show the equivalences of c) and d) we use \eqref{rics} and \eqref{nnewt} to  write \eqref{fmax} in the form
\begin{equation}\label{fmaxf}
\Delta\Big(Scal^g-\frac5{12}||T||^2 \Big)+\theta_j\sb_j\Big( Scal^g-\frac1{12}||T||^2|\Big)=||Ric||^2\ge 0.
\end{equation}
The strong maximum principle applied to \eqref{fmaxf} imply c) is equivalent to d) in the same way as above.

The proof of Theorem~\ref{closTt} is completed.
\end{proof}
The Corollary~\ref{parallel} follows from  the observation  that if $Scal^g=0$ then $Ric=0=Scal$ by Theorem~\ref{closTt}. Then \eqref{rics} gives $||T||^2=0$. Hence, $\LC\p=0$

\subsection{Steady generalized  Ricci solitons}
It is shown in \cite[Proposition~4.28]{GFS} that a Riemannian manifold $(M,g,T)$ with a closed 3-form $T$  is a  steady generalized  Ricci soliton 
if there exists a vector field $X$ and a two-form $B$ such that it is a solution to the equations 
\begin{equation}\label{gein3}
Ric^g=\frac14T^2-\frac12{\cal{L}}_X g, \qquad \delta T=B,
\end{equation}
for $ B$ satisfying $d(B+X\lrcorner T)=0.$
In particular $\Delta T=-{\cal{L}}_XT$, where $\Delta=d\d+\d d$ is the Laplace operator.

We also  recall  the equivalent formulation \cite[Definition~4.31]{GFS} that a compact Riemannian manifold $(M,g,T)$ with a closed 3-form $T$ is a steady generalized   Ricci soliton  if there exists a vector field $X$ and a closed 2-form $k$ such that
\begin{equation}\label{gein2}
Ric^g=\frac14T^2-\frac12{\cal{L}}_Xg, \qquad \delta T=-X\lrcorner T-2k, \qquad dT=0.
\end{equation}
If the vector field $X$ is a gradient of a smooth function $f$ then one has the notion of a generalized gradient Ricci soliton.
\begin{prop}\label{grsol}
Let $(M,\p)$ be an integrable $G_2$ manifold with closed torsion form, $dT=0$. 

Then it is a steady generalized   Ricci soliton with $X=\theta$ and $B=d\theta-\theta\lrcorner T$.
\end{prop}
\begin{proof}
As we identify the vector field $X$ with its corresponding 1-form via the metric, we have using \eqref{tsym}
\begin{equation}\label{acy1}
dX_{ij}=\LC_iX_j-\LC_jX_i=\sb_iX_j-\sb_jX_i+X_sT_{sij}.
\end{equation}
In view of \eqref{rics}, \eqref{tsym} and \eqref{acy1} we write the first equation in \eqref{gein3} in the form
\begin{equation}\label{acy2}
Ric_{ij}=-\frac12\delta T_{ij}-\frac12(\sb_iX_j+\sb_jX_i)=-\frac12\delta T_{ij}-\sb_iX_j+\frac12dX_{ij}-\frac12X_sT_{sij}.
\end{equation}
Set $X=\theta, \quad B=d\theta-\theta\lrcorner T$  and use \eqref{clos1} to get $\d T=B$ and \eqref{acy2} is trivially satisfied. Hence, \eqref{gein3} holds since $d(B+\theta\lrcorner T)=d^2\theta=0$.
\end{proof}
One fundamental consequence of Perelman's energy formula for Ricci flow is that compact steady solitons for Ricci flow are automatically gradient.  Adapting these energy functionals to generalized Ricci
flow it is proved in  \cite[Chapter~6]{GFS}  that steady generalized Ricci solitons on compact manifolds are automatically gradient and the 2-form $k = 0$  \cite[Corollary~6.11]{GFS}, 
i.e there exists a smooth function $f$ such that $X=grad(f)$ and \eqref{gein2} takes the form
\begin{equation}\label{gein1}
Ric^g_{ij}=\frac14T^2_{ij}-\LC_i\LC_j f, \qquad \delta T_{ij}=-df_sT_{sij}, \qquad dT=0.
\end{equation}
The smooth function $f$ is determined with $u=\exp(-\frac12f)$ where $u$ is  the first eigenfunction of the
Schr\"odinger operator, (see \cite[Lemma~6.3, Corollary~6.10, Corollary~6.11]{GFS}):
\begin{equation}\label{schro}4\Delta +Scal^g-\frac1{12}||T||^2=4\Delta +Scal+\frac16||T||^2.
\end{equation}
Note that we use here the Hodge Laplacian which differs with a sign from the Laplace-Beltrami operator used in  \cite[Lemma~6.3, Corollary~6.10, Corollary~6.11]{GFS}).

 In terms of the torsion connection \eqref{gein1} can be written in the form (see \cite{IS})
\begin{equation}\label{gein4}
Ric_{ij}=-\sb_i\sb_j f, \qquad \d T_{ij}=-df_sT_{sij}, \qquad dT=0.
\end{equation}
This combined with Proposition~\ref{grsol}  yield 
\begin{thrm}\label{nnnew}
A compact  integrable $G_2$ manifold  $(M,\p)$ with closed torsion is a steady generalized gradient Ricci soliton, i.e. there exists a smooth function $f$ on  $(M,\p)$ such that the equivalent equations \eqref{gein1} and \eqref{gein4} hold.
\end{thrm}
We also have
\begin{thrm}\label{inf}
Let $(M,\p)$ be a compact  integrable $G_2$ manifold with closed torsion, $dT=0$. The next two conditions are equivalent.
 \begin{itemize}
\item[a).] $(M,\p)$  is a steady generalized gradient Ricci soliton;
\item[b).]  For $f$ determined  by the first eigenfunction $u$ of the  Schr\"odinger operator \eqref{schro} with  $u=\exp(-\frac12f)$, the vector field $$V=\theta-df \quad  is \quad \sb-parallel, \quad \sb V=0.$$
\end{itemize}
Consequently, on any compact   integrable $G_2$ manifold with closed torsion there exists a $\sb$-parallel vector field $V$  which determines $d\theta$ and  preserves the $G_2$ structure $(g,\p)$  and the torsion 3-form $T$,
\begin{equation}\label{vsu3}
d\theta=V\lrcorner T, \quad {\cal{L}}_Vg={\cal{L}}_V\p={\cal{L}}_VT=0.
\end{equation}
If $V$ vanishes at one point then $V=0$ and, if in addition, the $G_2$ structure is strictly integable, $\p\wedge d\p=0$, then $T=0$, the compact $G_2$ manifold is parallel, $\LC\p=0$ and $f=const$.
\end{thrm}
\begin{proof}
To prove b) follows from a) observe that the  condition $dT=0$ and \eqref{rics} imply 
$$Ric=-\sb\theta, \quad -\d T=-d^{\sb}\theta=-d\theta+\theta\lrcorner T$$ by Theorem~\ref{closT} and \eqref{rics}. The latter  combined with \eqref{gein4} yield \[\sb(\theta-df)=0,\qquad d\theta=V\lrcorner T.\] 

For the converse, b)  yields $\sb\theta=\sb df$, which, combined  with Theorem~\ref{closT}  and \eqref{rics}  gives
\[Ric_{ij}=-\sb_i\theta_j=-\sb_i\sb_jf,\quad -\d T_{ij}=Ric_{ij}-Ric_{ji}=df_sT_{sij}.\]
since $0=d^2f=\LC_i\LC_jf-\LC_j\LC_if=\sb_i\sb_jf-\sb_j\sb_if+df_sT_{sij}$. 
Hence, \eqref{gein4} holds which proves the equivalences between a) and  b).

By Theorem~\ref{nnnew}, any compact integrable $G_2$ manifold is a steady generalized  gradient Ricci soliton and therefore the $\sb$-parallel vector field $V$ do exists because of b).

We have $( {\cal{L}}_Vg)_{ij}=\LC_iV_j+\LC_jV_i=\sb_iV_j+\sb_jV_i$ because of \eqref{tsym} the symmetric parts of $\LC V$  and  $\sb V$ coincide. Now, the condition $\sb V=0$ yields  $( {\cal{L}}_Vg)=0$. Hence $V$ is Killing.

Using $\sb V=0$, we obtain from the definitions of the Lie derivative and the torsion   (see e.g. \cite{KN}) 
\begin{multline*}
( {\cal{L}}_V\p)(X,Y,Z)=V\p(X,Y,Z)-\p([V,X],Y,Z)-\p(X,[V,Y],Z)-\p(X,Y,[V,Z])\\=V\p(X,Y,Z)-\p(\sb_VX,Y,Z)-\p(X,\sb_VY,Z-\p(X,Y,\sb_VZ)\\+T(V,X,e_a)\p(e_a,Y,Z)+T(V,Y,e_a)\p(X,e_a,Z)+T(V,Z,e_a)\p(X,Y,e_a)\\=(\sb_V\p)(X,Y,Z)+d\theta(X,e_a)\p(e_a,Y,Z)+d\theta(Y,e_a)\p(e_a,Z,X)+d\theta(Z,e_a)\p(e_a,X,Y)=0,
\end{multline*}
where we apply the already proved identity $d\theta=V\lrcorner T$ to get the second equality, use $\sb \p=0$ and that $d\theta\in \frak{g}_2\cong\Lambda^2_{14}$  to achieve the last identity.   Applying the Cartan formula for the Lie derivative, we obtain using $dT=0$ and the already proved first equality in \eqref{vsu3} that
\[{\cal{L}}_VT=V\lrcorner(dT)+d(V\lrcorner T)=d^2\theta=0. \]
The proof of \eqref{vsu3} is completed.

If $V=0$ at one point then $V=0$ because $\sb V=0$ and then the Lee form $\theta=df$ is an exact form. This combined with $dT=0$ and $d\p\wedge\p=0$ implies $T=df=0$ since M is compact, a fact  well known in physics as no go result \cite{GMW,GMPW}, see also \cite[Theorem~4.7]{Wit}, \cite[Proposition~3.1]{CGFT}. 
\end{proof}
 We  generalized the mentioned above no go theorem   assuming weaker condition on $dT$. First, it follows from \eqref{ricg2} and \eqref{dtt} that the four form $dT\in\Lambda^4_7\oplus\Lambda^4_{27}$ exactly when the characteristic scalar curvature is equal to the codiferential of the Lee form, $Scal=\delta\theta$ and $dT\in\Lambda^4_{27}$ if and only if  
the $G_2$ structure is of constant type and  $Scal=\delta\theta$.
\begin{thrm}\label{exact}
Let $(M,\p)$ be a compact integrable  $G_2$ manifold with an exact Lee form, $\theta=df$. If the four form  $dT\in\Lambda^4_7\oplus\Lambda^4_{27}$  then the following integral identity holds
\begin{equation}\label{ex11}\int_Me^{-f}\Big(||T||^2-\frac16(d\p,\ph)^2 \Big)vol.=0.
\end{equation}
If $(M,\p)$ is a compact strictly integrable $G_2$ manifold with $\theta=df$ and $Scal=\d df$ 
then $T=0$ and  the $G_2$ manifold is parallel, $\LC\p=0$.
\begin{proof}
 Suppose $\theta=df$ for a smooth function $f$ and $Scal=\d df$. Then \eqref{ricg2} yields $dT_{iabc}\ph_{iabc}=0$ and we obtain from \eqref{g22} 
\begin{equation}\label{exac}
||T||^2-\frac16(d\p,\ph)^2=6\Delta f+6||df||^2=6e^{f}\Delta u,
\end{equation}
where $\Delta f$ is the Laplace operator $\Delta f=-\LC_i\LC_if=-\sb_i\sb_if$ acting on smooth function $f$ since the torsion of $\sb$ is  a 3-form and the smooth function $u=e^{-f}f$.

Multiply  \eqref{exac} with $e^{-f}$ and integrate the obtained equality on the compact $M$ to achieve \eqref{ex11}.

If the $G_2$ manifold is strictly integrable, $(d\p,\ph)=0$, then \eqref{ex11} 
implies $T=0$ and $\LC\p=0$.
\end{proof}
\end{thrm}

\subsection{Examples}\label{bi}
 Basic examples of compact integrable $G_2$ manifolds with closed torsion are provided with the group manifolds which have flat torsion connection.  
 More precisely, it is well known that any compact 7-dimensional Lie group equipped with a biinvariant metric  and a left-invariant $G_2$ structure $\p$, generating the biinvariant metric, together with the left-invariant flat Cartan connection having closed torsion 3-form $T=-[.,.],$ preserving the $G_2$ structure $\p$, is an
invariant $G_2$ structure with a closed torsion 3-form, $dT=0$. However, the corresponding Lee form can be  closed but not exact and the constant type can be different from zero.
 We describe here some examples.
 
 \begin{exam}We consider $G=S^3\times S^3\times S^1=SU(2)\times SU(2)\times S^1$ with the biinvariant metric and the left-invariant  $G_2$ structure $\p$ and the flat left invariant  Cartan  connection with closed  torsion $T=-[.,.]$  which preserves $\p$, (see e.g.  \cite[Example~12.1]{GIP})  and investigated in detail in \cite{FMR}. 

We take the following description from \cite[Proposition~6.2]{FMR}.  On the group $G=SU(2)\times SU(2)\times S^1$ with Lie algebra $g=\frak{su}(2)\oplus \frak{su}(2)\oplus\mathbb R$ and structure equations   
\[de_1=e_{23},\quad de_2=e_{31},\quad de_3=e_{12},\quad de_4=e_{56},\quad de_5=e_{64}\quad de_6=e_{45}, \quad de_7=0\]
one considers the  family of left-invariant $SU(3)$ structure $(F,\ps_t,\sp_t)$ on $\frak{su}(2)\oplus \frak{su}(2)$ defined by
\[F=e_{14}+e_{25}-e_{36},\quad \ps_t=\cos t\ps+\sin t\sp, \quad \sp_t=-\sin t\ps+\cos t\sp, \quad where\]
\[\ps=e_{123}+e_{156}-e_{246}-e_{345},\quad \sp=e_{456}+e_{234}-e_{135}-e_{126} \]
and the family of $G_2$ structures on $G=SU(2)\times SU(2)\times S^1$ defined by
\[\p_t=F \wedge e_7+\ps_t, \quad *\p_t=\frac12F\wedge F+\sp_t\wedge e_7.\]
One obtaines after short calculations 
\[ d\p_t\wedge\p_t=
7(\cos t+\sin t)vol., \quad 
 d*\p_t=
\frac12(\cos t-\sin t)F\wedge F\wedge e_7=(\cos t-\sin t)e_7\wedge*\p_t,\]
\[\theta_t=(\cos t-\sin t)e_7.\]
It is known  from \cite[Proposition~4.5, Proposition~6.2]{FMR} that the left-invariant $G_2$ structures 
$\p_t$ are integrable, induce the same biinvariant metric on the group $G$, have the same closed and co-closed torsion $T=e_{123}+e_{456}$ which is the product of the 3-forms $T=-g([,],.)$ of each factor $SU(2)$. The characteristic connection is the left-invariant Cartan connection with torsion $T=-[,].$ Moreover, $\p_{\frac{3\pi}4}$ is strictly integrable,
$d\p_{\frac{3\pi}4}\wedge\p_{\frac{3\pi}4}=0$ and has closed Lee form, 
$\theta_{\frac{3\pi}4}=-\sqrt{2}de_7$. The structure $\p_0$ is of constant type  $d\p_0\wedge\p_0=7vol$ with closed Lie form $\theta_0=e_7$ while  $\p_{\frac{\pi}4}$ is cocalibrated of constant type, $\delta\p_{\frac{\pi}4}=0$ and $d\p_{\frac{\pi}4}\wedge\p_{\frac{\pi}4}=7\sqrt{2}.vol$.
\end{exam}
\begin{exam}
If one takes the left-invariant $G_2$ structure on  the group $G=SU(2)\times SU(2)\times U(1)$ defined by  \eqref{g2} it is easy to get that it is strictly integrable with closed torsion 3-form $T=e_{123}+e_{456}$ and non closed Lee form $\theta=e_4-e_3, d\p\wedge\p=0$. 
\end{exam}
\begin{exam}\label{noflate} We take this example from \cite{PP} (see also \cite[Example~6.4]{FF}). Let $N^4$ denote a hyperK\"ahler 4-manifold with (self-dual) hyperK\"ahler triple given by $\omega_1,\omega_2,\omega_3$ and consider $M^7 = SU(2)\times N^4$ endowed with the product $G_2$-structure defined by $\p=\omega_1\wedge e_1+\omega_1\wedge e_1+\omega_1\wedge e_1-e_{123}$, where $e_1,e_2,e_3$ is the  left-invariant co-framing on $\frak{su}(2)$. Then $(M,\p)$ is a compact integrable $G_2$ non-flat space of constant type with vanishing Lee form and closed torsion \cite{PP}. Moreover, the characteristic Ricci tensor  vanishes \cite[Example~6.4]{FF} (c.f.  \eqref{ricg2} ).
\end{exam}
\begin{rmrk}\label{rrrr} {\rm To the best of our knowledge there are not known  compact  strictly integrable $G_2$ manifolds with closed torsion which are not a group manifold. It seems  very interesting to understand whether there exist compact  strictly integrable  $G_2$ manifolds with closed torsion and  non-flat characteristic connection.}
\end{rmrk}
We note that the compact assumption in the Remark~\ref{rrrr} seems to be essential since a non compact strictly integrable $G_2$ manifold with closed and co-closed torsion having non vanishing characteristic Ricci tensor, and therefore non-flat characteristic connection,  is constructed in \cite[Example~7.1]{FF}

\vskip 0.5cm
\noindent{\bf Acknowledgements} \vskip 0.1cm
\noindent 
We would like to thank Jeffrey Streets and Ilka Agricola for   useful remarks, comments and suggestions. We also thank to the anonymous referees for their valuable remarks and comments  which help us to improve the paper. We also thank the referees for paying our attention to Example~\ref{noflate}.

The research of S.I.  is partially  financed by the European Union-Next Generation EU, through the National Recovery and Resilience Plan of the Republic of Bulgaria, project N:
BG-RRP-2.004-0008-C01. The research of N.S.  is partially supported   by Contract KP-06-H72-1/05.12.2023 with the National Science Fund of Bulgaria, Contract 80-10-61 / 27.05.2025   with the Sofia University "St.Kl.Ohridski" and  the National Science Fund of Bulgaria, National Scientific Program ``VIHREN", Project KP-06-DV-7.

\vskip 0.5cm


\end{document}